\documentclass[11pt,reqno,a4paper]{amsart}

\usepackage{amsmath,amssymb,amsfonts,amsthm}
\usepackage{latexsym}
\usepackage[english]{babel}
\usepackage{showidx}
\usepackage{graphics}
\usepackage{amsmath}
\usepackage{amssymb}

\newcommand{\rl}{{\mathbb{R}}}
\newcommand{\R}{{\mathbb{R}}}
\newcommand{\dt}{\delta}
\newcommand{\ep}{\varepsilon}

\DeclareMathOperator{\diam}{diam}

\theoremstyle{plain}
\newtheorem{theorem}{Theorem}[section]
\newtheorem{corollary}[theorem]{Corollary}
\newtheorem{lemma}[theorem]{Lemma}
\newtheorem{proposition}[theorem]{Proposition}
\theoremstyle{definition}
\newtheorem{definition}[theorem]{Definition}
\newtheorem{example}[theorem]{Example}
\newtheorem{remark}[theorem]{Remark}
\numberwithin{equation}{section}

\newcounter{remark}[section]
\newcounter{example}[section]
\def\R{\mathbb R}
\def\ball{I\kern -.35em B}
\def\cball{\overline{I\kern -.35em B}}

\begin{document}
\title [Singular gradient flow] {Singular gradient flow of the distance function and homotopy equivalence}
       \date{September 21, 2011}
       \author{P. Albano
       \quad P. Cannarsa
       \quad Khai T. Nguyen
    \quad C. Sinestrari
       }
      \address{Dipartimento di Matematica\\ 
Universit\`a di Bologna\\
Piazza di Porta San Donato 5\\ 40127 Bologna, ITALY} 
              \email{albano@dm.unibo.it}  
       \address{Dipartimento di Matematica\\ Universit\`a di Roma `Tor Vergata'\\ Via della
 Ricerca Scientifica 1\\ 00133 Roma, ITALY}
 \email{cannarsa@mat.uniroma2.it}
  \address{ Dipartimento di Matematica Pura ed Applicata\\
Universit\`a di Padova\\
 Via Trieste 63\\
35121 Padova, ITALY}
\email{khai@math.unipd.it}
 \address{Dipartimento di Matematica\\ Universit\`a di Roma `Tor Vergata'\\ Via della
 Ricerca Scientifica 1\\ 00133 Roma, ITALY}
\email{sinestra@mat.uniroma2.it}


\begin{abstract}
It is a generally shared opinion that significant information about the topology of a bounded domain $\Omega $ of  a  riemannian manifold  $M$ is encoded into the properties of the distance, $d_{\partial\Omega}$,
from the boundary of $\Omega$. To confirm such an idea we propose an approach based on the invariance of the singular set of the distance function  with respect to the generalized gradient flow of of $d_{\partial\Omega}$.
As an application, we deduce that such a singular set  has the same homotopy type as $\Omega$. 
\end{abstract}

\maketitle

\section{Introduction}

The distance function from the boundary of a bounded open set $\Omega\subset\R^n$ is a well-known object.
Defined as
\begin{equation*}
d_{\partial\Omega}(x)=\min_{y\in\partial\Omega}|y-x|\qquad\forall x\in \R^n\,,
\end{equation*}
it is a nonsmooth function that occurs in  different types of context. On account of its very definition, the distance retains some of the smoothness properties of the euclidean norm. For instance, $d_{\partial\Omega}$ is Lipschitz continuous in $\R^n$ with Lipschitz seminorm equal to one, and locally semiconcave in $\Omega$, that is, for every convex compact set $G\subset\Omega$ there is a constant $K_G\in\R$ such that $x\mapsto d_{\partial\Omega}(x)-K_G|x|^2/2$ is concave on $G$.
Moreover, the square of the distance is semiconcave with the uniform constant $K=2$ in the whole space, that is, $x\mapsto d_{\partial\Omega}^2(x)-|x|^2$ is concave in $\R^n$.
Consequently, $d_{\partial\Omega}$ is differentiable almost everywhere in $\R^n$. In particular, the distance is differentiable in $\Omega\setminus\Sigma$, where $\Sigma\subset\Omega$ is a set of Lebesgue measure zero called the singular set of $d_{\partial\Omega}$. In fact, general properties of semiconcave functions ensure that $\Sigma$ is countably $(n-1)$-rectifiable,
see \cite{AAC}.



The structure of the singular set has also been studied from the viewpoint of propagation of singularities. 
For general semiconcave functions, geometric conditions ensuring the propagation of singularities were 
obtained in \cite{AC1}. Subsequently, in \cite{AC2} and \cite{CY}, singularity-propagation results were derived for  viscosity solutions of the Hamilton-Jacobi
equation
\begin{equation}\label{intro:HJ}
F(x,u,Du)=0\quad \mbox{in}\quad\Omega\,,
\end{equation}
with $F(x,u,p)$ convex in $p$. More precisely, given a singular point $x_0$,
which is not critical for $u$, one can show the existence of a nonconstant lipschitz arc $\gamma
:[0,\sigma [\to \R^n$, starting at $x_0$, which consists of points where $u$ fails to be differentiable. Moreover, $\gamma$ is a generalized characteristic of \eqref{intro:HJ}, that is,  a solution of the
differential inclusion 
\begin{equation} 
\label{eq:GenCar}
\gamma' (t)\in \text{ co } D_pF(\gamma (t), u(\gamma (t)), D^+u(\gamma (t)))\,,
\end{equation}
where $D^+u$ denotes the superdifferential of $u$ while `co' stands for `convex hull'.
Observe that the above settings include the distance function, which solves the eikonal equation $|Dd_{\partial\Omega}|^2=1$. In this case, the differential inclusion \eqref{eq:GenCar} reduces to the gradient flow $\gamma' \in D^+d_{\partial\Omega}(\gamma)$---up to rescaling.

The propagation of singularities along characteristics is a well-studied property of solutions to linear hyperbolic equations. In~\cite{D}, for scalar hyperbolic conservation laws in one space dimension, Dafermos  observed that singular arcs could be regarded as generalized solutions of the same differential equation governing the dynamic of classical characteristics. 

The above considerations can be naturally extended to an open subset $\Omega$
of a riemannian manifold $M$. Here, in a local coordinate chart, the eikonal equation
takes the form
\begin{equation}
\label{eq:eikointro}
\langle A^{-1}(x)Du(x),Du(x)\rangle =1\,,
\end{equation}
 where 
$A(x)$ is related to the riemannian scalar product $g_x$ on the tangent space $T_xM$ by the formula
\begin{equation*}
g_x(\xi,\zeta )=\langle A(x)\xi ,\zeta \rangle \qquad\forall \xi,\zeta \in T_xM\,.
\end{equation*}
In this case, the equation of  generalized
characteristics is 
\begin{equation} 
\label{eq:GenGraFlo} 
\gamma '(t)\in A^{-1}(\gamma (t))D^+u(\gamma (t))\,,
\end{equation} 
which will be referred to as  the generalized gradient flow.

The gradient flow of the distance function on a manifold has often been  used in riemannian geometry as a tool for topological applications  in connection with Toponogov's theorem, starting from the seminal paper~\cite{GS} by Grove and Shiohama. A survey of the main results obtained by such techniques can be found in Chapter 11 of \cite{Petersen}. However, these authors' approach differs from the one considered here, in that they used a regularization of the gradient flow of the distance which admits smooth solutions.

In the aforementioned paper \cite{D}, singularities are shown to propagate  along a generalized characteristic, forward in time up to infinity. Intuitively speaking, such a behaviour is related to the well-known interpretation of the entropy condition for solutions of conservation laws (with a convex flux), which ensures that characteristics can only go inside a singularity: a characteristic which enters the singular set remains ``trapped'' there. 

Although  the two-dimensional structure of the problem is essential  for the proof of \cite{D}, one may wonder whether the same property holds for singular arcs of the distance function in arbitrary dimension.  The results of this paper give a precise response to such a question showing that $\Sigma$ is invariant for the generalized gradient flow, in the sense that any solution of \eqref{eq:GenGraFlo}, with $\gamma(0)\in\Sigma$, satisfies $\gamma(t)\in \Sigma$ for every $t\ge 0$. We note that such a property allows $\gamma$ to become constant at some time when a critical point of $d$ has been reached. In order to better explain the main idea of our approach, we will treat the euclidean case first, in Theorem~\ref{th:acn} below, and then extend the analysis to riemannian manifolds, in Theorem~\ref{th:acn2}.

The strategy of proof of our main results can be easily summarized, at an intuitive level, assuming that the arc $\gamma$ is smooth. For the euclidean distance in $\R^n$, basic results from \cite{AC2} ensure  that $\gamma(s)\in\Sigma$ if and only if $|\gamma'(s)|<1$. On the other hand, using properties of semiconcave functions we manage to show that the speed of a singular generalized characteristic satisfies the `logistic' differential inequality
\begin{equation}\label{eq:DI}
\frac{d}{ds}|\gamma'(s)|^2 \leq \frac{2}{d_{\partial\Omega}(\gamma((s))}|\gamma'(s)|^2\Big(\frac K2-|\gamma'(s)|^2\Big)\,,
\end{equation}
where $K$ is a semiconcavity constant for $d_{\partial\Omega}^2$. 
Since the optimal constant for the distance function is $K=2$, the above inequality yields $|\gamma'(s)|<1$ if the same is true for $s=0$, forcing $\gamma(s)\in\Sigma$. 

Surprisingly, the above method works in the riemannian case as well, even though $d_{\partial\Omega}^2$ fails to be semiconcave  with $K=2$, in general. One observation which is crucial for the extension is that, in fact, our idea is based on the existence of a nonlinear transform of the distance function which satisfies a  suitable differential constraint. It turns out that $\cosh (\alpha  \, d_{\partial\Omega})$, where $\alpha>0$ is a suitable constant depending on the sectional curvatures of $M$, has the required properties, see Theorem \ref{semman} below.
In addition, the proof requires the use of some standard tools in riemannian geometry, such as  parallel transport and properties of geodesics.

It is interesting to remark that semiconcave functions, and in particular the  distance function, have been studied in  different domains, often independently. Applications of semiconcavity can be found in nonlinear partial differential equations (\cite{Kruzkhov}, \cite{Krylov}), riemannian manifolds and Alexandrov spaces~(\cite{Perelman}, \cite{Petrunin}), control theory~(\cite{CS}), and
 optimal mass transportation~(\cite{Villani}).
Moreover, the singular set of the distance function is closely related to the cut-locus of the boundary of $\Omega$---a widely studied object in riemannian geometry,  mostly in the case of the distance from a single point (see e.g.\cite{Berger}).

More recently, interest in the properties of the singular set of $d_{\partial\Omega}$ arose in applied domains such as 
computer science (see \cite{CP,L} and the references therein). Several authors have studied, at increasing levels of generality, the homotopy equivalence between $\Sigma$ and $\Omega$.  If the boundary of $\Omega$ is smooth (or piecewise smooth in dimension $2$), then a homotopy can be constructed by moving every point outside $\Sigma$ by the smooth gradient flow of the distance function until it reaches $\Sigma$. However, this method only works if the distance from a regular point $p$ to $\Sigma$ along the gradient of the distance---the so called normal distance from $\Sigma$---is a continuous function of $p$. If $\Omega$ is a nonsmooth set of dimension at least $3$, then this fails to be true and the construction of the homotopy becomes more involved \cite{L}. An essential step of the procedure is to extend the gradient flow of the distance function past a singular point, ensuring that the corresponding generalized characteristic stays singular for all time. This is why our  result about the invariance of the singular set under such a flow can be used to provide an easy proof of the result of \cite{L} and, more importantly, generalize the homotopy equivalence between $\Sigma$ and $\Omega$ to the case of a complete riemannian manifold (see Theorem~\ref{theorem:1} below). 

As a possible application of the analysis developed in this paper, we would like to mention optimal exit time problems  in $\R^n$ that can be subsumed by the riemannian setting. Specifically, we study a  class of time optimal control problems for which the minimum time function $T(x)$, which measures the minimum time needed to steer a point $x\in\Omega$ to $\partial\Omega$, can be interpreted as a riemannian distance function (Example~\ref{ex:mintime} below). As a consequence, we  deduce the homotopy equivalence between $\Omega$ and the singular set of $T(\cdot)$, a result that would be hard to derive keeping the reasoning confined to euclidean space.

This paper is organized as follows. First, in section~\ref{se:notation}, we introduce the essentials of our notation. Then,
in section~\ref{se:eucli}, we show the invariance of $\Sigma$ in an open subset of $\R^n$, and  generalize this result to riemannian settings in section~\ref{se:riema}. 
 Finally, in section~\ref{se:homo}, we prove  homotopy equivalence  and discuss applications to optimal control problems.

\section{Notation}\label{se:notation}

Given $x,y \in \rl^n$, we denote by $\langle x,y \rangle$ and by  $|x|$ the euclidean scalar product  and norm respectively. We set $d(x,y):=|x-y|$ for the distance between two points.

If $M$ is a riemannian manifold and $x \in M$, we denote by $T_xM$ the tangent space to $M$ at $x$ and by $T^*_xM$ the cotangent space. There is a canonical isomorphism between the two spaces given by the scalar product. For simplicity, we use the same symbol $\langle \cdot,\cdot \rangle$ to denote the scalar product of two vectors of $T_xM$, or of two elements of $T^*_xM$, or the pairing of a form in $T^*_xM$ and a vector in $T_xM$. The riemannian distance between two points $x,y \in M$ will be again denoted by $d(x,y)$.

For a function $u$ defined in a subset of $\rl^n$, we denote by $Du$ and $D^2u$ the gradient and  hessian of $u$ where they exist. If $u$ is a function on a riemannian manifold $M$, we denote by $Du(x) \in T_xM$ its gradient and by $du(x) \in T^*_xM$ its differential; although the two notions are equivalent via the isomorphism recalled above, it will be convenient for our purposes to keep them distinct. In addition, we denote by $D^2u$ the hessian of $u$, interpreted as a linear operator from $T_xM$ to itself, as for example in \cite[Ch. 14]{Villani}.

If $u$ is a Lipschitz continuous function defined in a riemannian manifold, we denote by $\Sigma(u)$ the set of points where $u$ is not differentiable, and we call such a set the {\em singular set} of $u$. By Rademacher's theorem, $\Sigma(u)$ has Lebesgue measure zero.

If $C$ is a nonempty closed subset of a riemannian manifold $M$ (in particular $M$ could be $\rl^n$), we denote by $d_C(x)$ the {\em distance function} from $C$, defined as
\begin{equation}\label{e.df}
d_C(x)=\min_{y \in C}d(y,x).
\end{equation}
It is well known that $d_C(x)$ is Lipschitz continuous with constant $1$ and that $|D d_C(x)|=1$ at every $x$ where $d_C$ is differentiable.

\section{The euclidean case}\label{se:eucli}

We first recall some properties of semiconcave functions. In this section, $\Omega$ will be an open set in $\rl^n$. 

\begin{definition}
A function $u:\Omega \to \rl$ is called {\em semiconcave} if there exists $K \geq 0$ such that
$$
t u(x) + (1-t) u(y) - u(t x + (1-t)y) \leq t(1-t)K \frac{|x-y|^2}{2}
$$
for any $x,y \in \Omega$ such that the segment from $x$ to $y$ is contained in $\Omega$, and for any $t \in [0,1]$. We call $K$ a {\em semiconcavity constant} for $u$ in $\Omega$. We say that $u$ is locally semiconcave in $\Omega$ if it is semiconcave on any subset $A \subset \subset \Omega$.
\end{definition}

It is easy to see that $u$ is semiconcave with constant $K$ if and only if the function $u(x)-\frac{K}{2}|x|^2$ is concave or if $D^2u \leq K \, \mbox{\it Id}$ in the sense of distributions, where {\it Id} denotes the identity matrix.

The {\em (Fr\'echet) superdifferential} of a function $u:\Omega \to \rl$ at a point $x \in \Omega$ is defined as the set
$$
D^+u(x)=\left\{ p \in \rl^n ~:~ \limsup_{y \to x} \frac{u(y)-u(x)-\langle p,y-x \rangle}{|y-x|} \leq 0 \right\}.
$$
In the case of a semiconcave function, the superdifferential enjoys the following properties; the proofs can be found in any textbook on convex analysis or in Chapter 3 of \cite{CS}.

\begin{proposition}\label{p.semic}
Let $u:\Omega \to \rl$ be semiconcave.
\begin{enumerate}
\item[(i)]
The function $u$ is locally Lipschitz continuous in $\Omega$ and differentiable almost everywhere. 
\item[(ii)]
The superdifferential $D^+u(x)$ is nonempty for all $x \in \Omega$. It is a singleton if and only if $u$ is differentiable at $x$, and in this case we have $D^+u(x)=\{Du(x)\},$ where $Du(x)$ is the standard gradient.
\item[(iii)]
For any $x,y \in \Omega$ such that the segment from $x$ to $y$ is contained in $\Omega$, for any $p \in D^+u(x)$ and $q \in D^+u(y)$, we have
\begin{equation}
\langle q-p, y-x \rangle \leq K|y-x|^2,
\end{equation}
where $K$ is a semiconcavity constant of $u$.
\item[(iv)]
Given $\{x_n\} \subset \Omega$ such that $x_n \to \bar x \in \Omega$ and $p_n \in D^+u(x_n)$ such that $p_n \to \bar p$, we have that $\bar p \in D^+u(\bar x)$.
\end{enumerate}
\end{proposition}

The distance function $d_C$ provides an example of semiconcave function, as we recall here.
  
\begin{proposition}\label{p.d2}
Given any nonempty closed set $C \subset \rl^n$, the distance function $d_C$ is locally semiconcave on $\rl^n \setminus C$. In addition, the squared distance function $d_C^2(\cdot)$ is semiconcave on all $\rl^n$ with constant $K=2$. Moreover, we have
\begin{equation}\label{e.mond2}
\langle d_C(x)p-d_C(y)q, x-y\rangle \le | x-y|^2
\end{equation}
for all $x,y \in \rl^n, p\in D^+d(x)$ and $q\in D^+d(y)$.
\end{proposition}
\begin{proof} The statements about the semiconcavity of $d_C$ and $d_C^2$ follow from an easy computation, see e.g. \cite[Proposition 2.2.2]{CS}. Estimate (\ref{e.mond2}) then follows from Proposition \ref{p.semic}(iii), observing that $D^+d^2_C(x)=2d_C(x) D^+d_C(x).$

\end{proof}

In the following we will consider semiconcave functions $u:\Omega \to \rl$ which solve an equation of the form 
\begin{equation}\label{e.ik}
\langle A^{-1}(x) Du(x), Du(x) \rangle =1, \qquad x \in \Omega,
\end{equation}
where $A(x)$ a symmetric positive definite $n \times n$ matrix with $C^1$ dependence on $x \in \Omega$. The formulation of the equation with $A^{-1}$, rather than $A$, is more convenient in view of the application to the distance function on riemannian manifolds in the next section, where $A(\cdot)$ will be the matrix associated with the metric on the tangent space.

It is well known (see e.g. \cite[Prop. 5.3.1]{CS}) that, if $u:\Omega \to \rl$ is semiconcave, then the following properties are equivalent:
\begin{itemize}
\item $u$ satisfies \eqref{e.ik} at every $x \in \Omega$ where $Du(x)$ exists;
\item for every $x \in \Omega$ and $p \in D^+u(x)$ we have $\langle A^{-1}(x) p,p \rangle \leq 1$;
\item $u$ is a viscosity solution of \eqref{e.ik} (in the sense of \cite{CEL}).
\end{itemize}
Throughout the paper, we call a {\em solution} of \eqref{e.ik} a locally semiconcave function satisfying the above properties.

Given a solution of \eqref{e.ik}, we consider the differential inclusion
\begin{equation}\label{eq:genc1}
\gamma' (t)\in A^{-1}(\gamma(t)) D^+u(\gamma (t)).
\end{equation}
A Lipschitz arc $\gamma:[0,t_0] \to \Omega$ is called a solution to the above problem if,  for a.e. $t \in [0,t_0]$, it  satisfies $\gamma' (t)= A^{-1}(\gamma(t)) p(t)$ for some element $p(t) \in D^+u(\gamma (t))$. Such an arc will also be called a {\em generalized characteristic} of equation \eqref{e.ik} associated with $u$.

We now recall some properties of generalized characteristics. The main part of the statement (in particular claim (iv) about the propagation of singularities) follows from the results first proved in \cite{AC2} and then obtained with a simpler approach in \cite{CY,Y}. For the convenience of the reader, we include the proof of some additional properties which were not explicitly observed in the above references.

\begin{theorem}\label{th:pro1} 
Let $u:\Omega \to \rl$ be a solution of \eqref{e.ik}. Then, for every $x_0 \in\Omega$ there exists $t_0 >0$ and a unique Lipschitz continuous arc $\gamma :[0,t_0 [\rightarrow \Omega$ which satisfies \eqref{eq:genc1} and the initial condition $\gamma(0)=x_0$. In addition, the right derivative $\gamma_+'(t)$ exists for {\em every} $t \in [0,t_0[ \,$, and $p(t):=\gamma_+'(t)$ has the following properties:
\begin{enumerate}
\item[(i)]
$p(t) \in A^{-1}(\gamma(t)) D^+u(\gamma(t))$ for every $t\in [0,t_0 [ \,$ and
\begin{equation}\label{e.mi}
\langle  p(t), A(\gamma(t)) p(t) \rangle \leq \langle  q, A(\gamma(t)) q \rangle, 
\;
\forall \, q \in A^{-1}(\gamma(t))D^+u(\gamma(t)).
\end{equation}
\item[(ii)]
$p(t)$ is continuous from the right for every $t\in [0,t_0 [ \,$ and, for all points $t^*$ where it is discontinuous, we have
\begin{equation}\label{e.lsc}
\liminf_{t \to t^*-}\langle  p(t), A(\gamma(t)) p(t) \rangle \geq \langle  p(t^*), A(\gamma(t^*)) p(t^*) \rangle. 
\end{equation}
\item[(iii)]
For any $t \in [0,t_0[\,$,   $\gamma(t) \in \Sigma(u)$ if and only if $\langle A(\gamma(t)) p(t), p(t) \rangle <1$.
\item[(iv)]
If $x_0 \in \Sigma(u)$, then there exists $\sigma\in]0,t_0]$ such that $\gamma(t) \in \Sigma(u)$ for all $t \in [0,\sigma]$.
\item[(v)]
For all $t \in [0,t_0[ \,$ we have
\begin{equation}\label{eqmon}
\frac {d}{dt^+} u(\gamma(t)) = \langle A(\gamma(t)) p(t), p(t) \rangle\,,
\end{equation}
where the symbol $\frac {d}{dt^+}$ denotes the derivative from the right.
\end{enumerate}
\end{theorem}

Before proving the theorem, we give an elementary continuous dependence result for generalized characteristics.

\begin{lemma} 
\label{th:unique}
Under the above assumptions, for any $U \subset \subset \Omega$ there exist $C>0$ and $t_0 >0$ such that, if $x,y \in U$ and $\gamma_x,\gamma_y$ are solutions of $(\ref{eq:genc1})$ with initial conditions $\gamma_x(0)=x$ and $\gamma_y(0)=y$ respectively, then
\begin{equation}\label{eq:unique}
| \gamma_x(t)-\gamma_y(t)|\le C |x-y|, \qquad t \in [0,t_0].
\end{equation}
\end{lemma}
\begin{proof}
Let us take any ball $B \subset \subset \Omega$ and points $x_1,x_2 \in B$. Given any $p_1 \in D^+u(x_1)$ and $p_2 \in D^+u(x_2)$, we have
\begin{multline*}
\langle A^{-1}(x_1)p_1-A^{-1}(x_2)p_2,A(x_1)(x_1-x_2)\rangle = \langle p_1-A(x_1)A^{-1}(x_2)p_2,
x_1-x_2\rangle 
\\
=\langle p_1-p_2,x_1-x_2\rangle+ \langle (I-A(x_1)A(x_2)^{-1})p_2,x_1-x_2\rangle
\le C_1|x_1-x_2|^2,
\end{multline*}
for some $C_1>0$, where we have used property (iii) of  Proposition \ref{p.semic} and
the Lipschitz continuity of the map $\overline{B}\ni x\mapsto A^{-1}(x)$.

Let now $\gamma_x,\gamma_y$ be two solutions of (\ref{eq:genc1}) contained in $B$. Using the above estimate and the nondegeneracy of the matrix $A(\cdot)$ we find that
\begin{eqnarray*}
\lefteqn{\frac{d}{dt} \langle A(\gamma_x(t))
(\gamma_x(t)-\gamma_y(t)),\gamma_x(t)-\gamma_y(t)\rangle }\\
& = &
 \left \langle \left( \frac{d}{dt} A(\gamma_x(t)) \right)
(\gamma_x(t)-\gamma_y(t)),\gamma_x(t)-\gamma_y(t) \right\rangle \\
&& +
2 \langle \gamma'_x(t)- \gamma'_y(t)), A(\gamma_x(t))(\gamma_x(t)-\gamma_y(t))\rangle \\
& \le & C_2
|\gamma_x(t)-\gamma_y(t)|^2  \\
& \le &
C_3 \langle A(\gamma_x(t))
(\gamma_x(t)-\gamma_y(t)),\gamma_x(t)-\gamma_y(t)\rangle . 
\end{eqnarray*}
The Gronwall inequality yields 
$$
 \langle A(\gamma_x(t))
(\gamma_x(t)-\gamma_y(t)),\gamma_x(t)-\gamma_y(t)\rangle \leq C_4|x-y|^2,
$$
which implies the conclusion. \end{proof}

\begin{proof}{\it of Theorem \ref{th:pro1}} \ 
The existence of an arc satisfying (\ref{eq:genc1}) is proved in \cite[Theorem 3.2]{CY} (see also \cite{AC2}), while  uniqueness follows from Lemma \ref{th:unique}. 

The existence and  continuity from the right of $\gamma'_+(t)$ for every $t$, as well as properties (i) and (iv), follow from Corollaries 3.3 and 3.4 in \cite{CY}. Actually, those results require  the additional assumption that $0 \notin D^+u(x_0)$; however, if $0 \in D^+u(x_0)$ then the arc $\gamma$ is the constant one $\gamma(t) \equiv x_0$, and all the properties of our statement are trivially satisfied. 

To prove (\ref{e.lsc}), let us pick any sequence $t_n \uparrow t^*$ such that $\lim_{n \to 
\infty} p(t_n)$ exists. By (i), we have that $A(\gamma(t_n))p(t_n) \in D^+u(\gamma(t_n))$ for any $n$. Therefore, setting $p_-=\lim_{n \to \infty} p(t_n)$, we have by Proposition \ref{p.semic}(iv) that $A(\gamma(t^*))p_- \in D^+u(\gamma(t^*))$. But then we obtain from (\ref{e.mi}) that 
$$
\langle  p_-, A(\gamma(t^*)) p_- \rangle \geq \langle  p(t^*), A(\gamma(t^*)) p(t^*) \rangle,
$$
which implies (\ref{e.lsc}) and completes the proof of (ii).

To prove (iii), suppose that $u$ is differentiable at $\gamma(t)$. Then, by part (i) and  Proposition \ref{p.semic}(ii), we have that $A(\gamma(t))p(t)=Du(\gamma(t))$. Therefore, since $u$ solves \eqref{e.ik}, $\langle p(t), A(\gamma(t))p(t) \rangle=1$. Conversely, if $u$ is not differentiable at $\gamma(t)$, then $D^+u(\gamma(t))$ contains more than one point.  As recalled previously, a solution $u$ to equation \eqref{e.ik} satisfies $\langle A(\gamma(t))q,q \rangle \leq 1$ for all $q \in D^+u(\gamma(t))$. Since $D^+u(\gamma(t))$ is a convex set containing more than one point and  $A$ is positive definite, we deduce that $\langle A(\gamma(t))q,q \rangle < 1$ for some $q \in D^+u(\gamma(t))$. Then (ii) implies that $\langle p(t), A(\gamma(t))p(t) \rangle<1$.

To prove property (v), we have to recall some details of the proof of the existence of the singular arc $\gamma$ given in \cite{CY}. The authors introduce there a family of smooth functions $u_k$ converging uniformly to $u$ with bounded Lipschitz constant. The arc $\gamma$ is then obtained as the uniform limit of a sequence of smooth arcs $\gamma_k$ which solve the equation $$\gamma'_k(t)=A^{-1}(\gamma_k(t))Du_k(\gamma_k(t)).$$
The above properties easily imply that we also have weak convergence $\gamma'_k \rightharpoonup \gamma'$ in $L^2([0,t_0], \rl^n)$. If we exploit the lower semicontinuity of convex functionals with respect to weak convergence, we obtain, for any $0 \leq t_1 <t_2 \leq t_0$, 
\begin{eqnarray*} \lefteqn{
\liminf_{k \to \infty} \int_{t_1}^{t_2} \langle A(\gamma_k(t)) \gamma'_k(t), \gamma'_k(t) \rangle dt}  \\
& = & \liminf_{k \to \infty} \int_{t_1}^{t_2} \langle A(\gamma(t)) \gamma'_k(t), \gamma'_k(t) \rangle dt \\
& &+
\lim_{k \to \infty} \int_{t_1}^{t_2} \langle [A(\gamma_k(t)) - A(\gamma(t))]\gamma'_k(t), \gamma'_k(t) \rangle dt
 \\
& \geq  &\int_{t_1}^{t_2} \langle A(\gamma(t)) \gamma'(t), \gamma'(t) \rangle dt,
\end{eqnarray*}
where we have also used the uniform convergence of $\gamma_k$ and the boundedness of $\gamma'_k$. It follows that
\begin{eqnarray*}
\lefteqn{u(\gamma(t_2))-u(\gamma(t_1)  =  \lim_{k \to \infty} u_k(\gamma_k(t_2)) - u_k(\gamma_k(t_1))} \\
 &= & \lim_{k \to \infty} \int_{t_1}^{t_2} \langle Du_k(\gamma_k(t)), \gamma'_k (t)\rangle \, dt 
 =  \lim_{k \to \infty} \int_{t_1}^{t_2} \langle A(\gamma_k(t)) \gamma'_k (t), \gamma'_k (t)\rangle \, dt\\
& \geq & \int_{t_1}^{t_2} \langle A(\gamma(t)) \gamma'(t), \gamma'(t) \rangle \, dt
= \int_{t_1}^{t_2} \langle A(\gamma(t)) p(t), p(t) \rangle \, dt,
\end{eqnarray*}
which implies 
$$
\liminf_{h \to 0^+} \frac{u(\gamma(t+h))-u(\gamma(t))}{h} \geq  \langle A(\gamma(t)) p(t), p(t) \rangle,
$$
for every $t \in [0,t_0[ \,$, by the right continuity of $p$. On the other hand, since $A(\gamma(t)) p(t) \in D^+u(\gamma(t))$, we also have
\begin{eqnarray*}
\limsup_{h \to 0^+} \frac{u(\gamma(t+h))-u(\gamma(t))}{h} & \leq & \lim_{h \to 0^+} 
\left\langle A(\gamma(t)) p(t), \frac{\gamma(t+h)-\gamma(t)}{h} \right\rangle \\
 & = & \langle A(\gamma(t)) p(t), p(t) \rangle.
\end{eqnarray*}
This concludes the proof. 
\end{proof}

\begin{corollary}\label{c.prop}
Let $\Omega \subset \R^n$ be an open set.
Then, for every $x\in\Omega$ there exists a
unique Lipschitz continuous arc $\gamma :[0,\infty [\rightarrow \Omega$
such that 
\begin{equation}
\label{eq:genc}
\gamma ' (t)\in D^+d_{\partial \Omega}(\gamma(t))\quad t\in [0,\infty [\text{ a.e.}
\qquad \gamma (0)=x. 
\end{equation}
In addition, for any $t_0 \geq 0$ such that $\gamma(t_0) \in \Sigma(d_{\partial \Omega})$ there exists $\sigma>0$ such that $\gamma(t) \in \Sigma(d_{\partial \Omega})$ for all $t \in [t_0,t_0+\sigma[\, $. 
Finally, the derivative from the right $\gamma'_+(t)$ exists for all $t \in [0,+\infty)$ and satisfies the properties described in Theorem \ref{th:pro1}, with $u(x)=d_{\partial \Omega}(x)$ and $A(x)\equiv Id$.
\end{corollary}
\begin{proof}
The statement follows directly from Theorem \ref{th:pro1}, provided we show that the maximal interval of existence of  $\gamma$ is $[0,+\infty[\,$. To see this we note that, if such an interval is $[0,T[\,$ with $T \neq +\infty$, then necessarily $\gamma(t)$ approaches $\partial \Omega$ as $t \to T$, that is, $d_{\partial \Omega}(\gamma(t)) \to 0$ as $t \to T$, in contrast with the property that $d_{\partial \Omega}(\gamma(t))$ is positive and nondecreasing in $t$ by (\ref{eqmon}). 
\end{proof}

We are now ready to prove the main result of this section.

\begin{theorem}
\label{th:acn} 
Let $\Omega \subset \R^n$ be an open set, let
$x\in\Omega$, and let $\gamma(\cdot)$ be the solution of  \eqref{eq:genc} given by Corollary \ref{c.prop}. If $\gamma(t_0) \in \Sigma(d_{\partial \Omega})$ for some $t_0 \geq 0$, then $\gamma(t) \in \Sigma(d_{\partial \Omega})$ for all $t \in [t_0,+\infty[$.
\end{theorem} 
\begin{proof}
For simplicity of notation, we suppose $t_0=0$. Set 
$$
p(t):=\gamma'_+(t), \qquad \delta (t):=d_{\partial \Omega}(\gamma (t))\qquad t\in [0,\infty [\, .
$$
From Theorem \ref{th:pro1} and Corollary \ref{c.prop} we know  that $p(t) \in D^+d_{\partial \Omega}(\gamma(t))$ for all $t$; in addition $\gamma(t) \in \Sigma(d_{\partial \Omega})$ and $|p(t)|<1$ for all $t>0$ in a right neighbourhood of $0$. Our aim is to show that $|p(t)|<1$ holds for every $t \geq 0$; by part (iii) of Theorem \ref{th:pro1}, this will prove our assertion.

Let $0\le s<t$. By Proposition \ref{p.d2}, we have 
\begin{multline}
\label{eq:e2}
\delta(t)\langle p(t)-p(s),\gamma (t)-\gamma(s)\rangle 
\\
\le |\gamma (t)-\gamma(s)|^2- (\delta(t)-\delta(s)) \langle p(s),\gamma (t)-\gamma(s)\rangle . 
\end{multline}
Now, set $t=s+h$ ($h>0$), 
$$
\gamma_h(s)=\frac{\gamma (s+h)-\gamma (s)}{h}\qquad \text{ and }\qquad  \delta_h(s)=\frac{\delta (s+h)-\delta (s)}{h}. 
$$
We observe that $|\gamma_h| \leq 1,|\delta_h| \leq 1$, since $\gamma$ and $\delta$ are both $1$-Lipschitz function. From (\ref{eq:e2}) we obtain
\begin{equation}
\label{eq:e3}
\left\langle \gamma'_h(s)
\, , \,  \gamma_h(s)\right\rangle \le \frac{1}{\delta(s+h)}[ |\gamma_h(s)|^2 - \delta_h(s) \langle p(s),\gamma_h(s)\rangle ].
\end{equation}
Here and in the rest of the proof, we use for simplicity the notation of the ordinary derivative to mean the derivative from the right of expressions involving $\gamma_h$.
Observe that, by Theorem \ref{th:pro1}, we have for all $s \geq 0$
\begin{equation}
\label{eq:e3b}
\lim_{h \to 0^+}\gamma_h(s)=p(s), \qquad \lim_{h \to 0^+} \delta_h(s)=|p(s)|^2. 
\end{equation}

From a heuristic point of view, it is useful to take the limit as $h \downarrow 0$ in (\ref{eq:e3}). We obtain 
\begin{equation}\label{e.int}
\frac{d}{ds}|p(s)|^2 \leq \frac{2}{\delta(s)}|p(s)|^2(1-|p(s)|^2).
\end{equation}
Such an inequality implies that, if $|p(t_0)|<1$, then $|p(s)|<1$ for all $s >t_0$. However, such a reasoning is only formal, because we cannot say anything about the differentiability of $p(\cdot)$. 
It is interesting to observe that the crucial constant $1$, in the expression $(1-|p(s)|^2)$ above, arises from the previous computations as $K/2$, where $K=2$ is the semiconcavity constant of $d_{\partial \Omega}^2$.

Although the above argument is not rigorous, it suggests that $\gamma_h(\cdot)$ can be estimated by a suitable adaptation of the separation of variables procedure which could be used to integrate \eqref{e.int}. To do this, let us first fix $\ep>0$ small. Since $0 \leq |\gamma_h(s)| \leq 1$ for all $s \geq 0, h>0$, we have that
\begin{equation}\label{e:ps}
(|\gamma_h(s)|^2+\ep)(1+\ep-|\gamma_h(s)|^2) \geq  \ep(1+\ep) >0, \quad s \geq 0, h>0.
\end{equation}
Thus, we can divide both sides of \eqref{eq:e3}  by the above expression to obtain
$$
\frac{(|\gamma_h(s)|^2)'}
{(|\gamma_h(s)|^2+\varepsilon) (1+\varepsilon-|\gamma_h(s)|^2)}
\le \frac{2}{\delta(s+h)}\, \frac{|\gamma_h(s)|^2-\delta_h(s)\langle p(s),\gamma_h(s)\rangle}{(|\gamma_h(s)|^2+\varepsilon) (1+\varepsilon-|\gamma_h(s)|^2)}\,.
$$
Integrating over $[0,t]$, we find
$$
\log \Big ( \frac{|\gamma_h(s)|^2+\varepsilon}{1+\varepsilon-|\gamma_h(s)|^2}\Big )  \Big|_0^t  \le F_{h,\ep}(t)\,,
$$ 
where 
$$
F_{h,\ep}(t)= (1+2\ep)\int_0^t\frac{2}{\delta(s+h)}\, \frac{|\gamma_h(s)|^2-\delta_h(s)\langle p(s),\gamma_h(s)\rangle}{(|\gamma_h(s)|^2+\varepsilon) (1+\varepsilon-|\gamma_h(s)|^2)}\ ds\,.
$$
Therefore, 
\begin{equation}
\label{eq:e5}
| \gamma_h(t)|^2 \le \frac{e^{F_{h,\ep}(t)} \xi_{h,\ep}(1+\varepsilon)-\varepsilon}{1+e^{F_{h,\ep}(t)} \xi_{h,\ep}} = 1+\ep  - \frac{1+2\ep}{1+e^{F_{h,\ep}(t)} \xi_{h,\ep}}\,,
\end{equation}
where 
$$
\xi_{h,\ep}:=\frac{|\gamma_h(0)|^2+\varepsilon }{1+\varepsilon-|\gamma_h(0)|^2 }\,.
$$
We now want to let first $h \downarrow 0$ and then $\ep \downarrow 0$. By (\ref{eq:e3b}) we find,
\begin{equation*}
\lim_{\ep \downarrow 0} \left( \lim_{h\downarrow 0}\xi_{h,\ep}\right) = \lim_{\ep \downarrow 0} \frac{|p(0)|^2+\varepsilon}{1+\varepsilon-|p(0)|^2}= \frac {|p(0)|^2}{1-|p(0)|^2}\,.
\end{equation*}
On the other hand, since by \eqref{e:ps} the integrand in the definition of $F_{h,\ep}$ is uniformly bounded in $h$, we obtain, again by (\ref{eq:e3b}),
\begin{eqnarray*}
\lim_{h\downarrow 0}F_{h,\ep}(t) & = & (1+2\varepsilon)\int_0^t \frac{2}{\delta(s)}\,\frac{|p(s)|^2(1-|p(s)|^2)}{(|p(s)|^2+\varepsilon) (1+\varepsilon-|p(s)|^2)}\,ds  \\
& \leq & (1+2\varepsilon)\int_0^t \frac{2}{\delta(s)}\,ds \,,
\end{eqnarray*}
which implies
\begin{equation*}
\limsup_{\ep \downarrow 0} \left( \lim_{h\downarrow 0}F_{h,\ep}(t)\right)=
\int_0^t \frac{2}{\delta(s)}\,ds\, =: \alpha (t)\,.
\end{equation*}
Thus, letting first $h\downarrow 0$ and then $\ep \downarrow 0$ in \eqref{eq:e5}, we conclude that 
\begin{equation}\label{eq:p_bound}
|p(t)|^2\le \frac{|p(0)|^2e^{\alpha (t)}}{|p(0)|^2 e^{\alpha (t)}
  +1-|p(0)|^2}\,.
\end{equation}
Therefore, $|p(t)|^2<1$ for every $t \geq 0$, and the proof is complete. 
\end{proof}

\section{The riemannian case}\hspace{5 mm}  \label{se:riema}

Let us now consider a complete riemannian manifold $M$, possibly noncompact. To extend to this framework the techniques of the previous section, we first need to recall some basic properties of  parallel transport and  geodesic curves. On $M$, there is a canonical notion of derivative of a vector field, called covariant derivative. Using this definition, a vector field is called {\em parallel} along a curve $\gamma$ if its derivative in direction $\gamma'(t)$ is zero for all $t$. If we have a curve $\gamma:[a,b] \to M$ and a vector $v \in T_{\gamma(a)}M$, there is a unique vector field $v(t) \in T_{\gamma(t)}M$, with $t \in [a,b]$, which is parallel along $\gamma$; such a field $v(t)$ is called the {\em parallel transport} of $v$ along $\gamma(t)$. Parallel transport preserves the scalar product and therefore gives an isometry between the tangent spaces at different points.
The {\em geodesics} on $M$ can be defined equivalently as the curves $\gamma$ such that the speed $\gamma'(t)$ is parallel along the curve $\gamma$ itself or as the curves which are stationary for the energy functional. Geodesics have constant speed and are curves of minimal length between two endpoints if these points are close enough to each other.

Given a point $x \in M$, we denote by $\exp_x(\cdot)$ the exponential map at $x$. We recall that, given a tangent vector $v \in T_x M$, $\exp_x(v)$ is the point reached at $t=1$ by the geodesic $\gamma(t)$  starting with $\gamma(0)=x$ and $\gamma'(0)=v$. If $f$ is a smooth function and $df(x) \in T^*_xM$ is its differential at $x$, we have
$$
f( \exp_x(v))-f(x) = \langle df(x),v \rangle + o(|v|), \qquad   v \in T_xM, \quad v \to 0.
$$

Let us now consider a function $u:M \to \rl$ not necessarily smooth. We say that $p \in T^*_xM$ belongs to $d^+u(x)$, the {\em superdifferential} of $u$ at $x$, if
$$
u( \exp_x(v))-u(x) \leq \langle p,v \rangle + o(|v|), \qquad v \in T_xM, \quad v \to 0.
$$
This is equivalent to saying that there exists a smooth function $f$ touching $u$ from above at $x$ such that $df(x)=p$. It is easy to see that, if $p \in d^+u(x)$ and if $\gamma:[-a,a] \to M$ is any smooth curve such that $\gamma(0)=x$ (not necessarily a geodesic), then
\begin{equation}
\label{e.curve}
\limsup_{h \to 0}  \frac {u(\gamma(h)) -u(\gamma(0))}{h} \leq \langle p, \gamma'(0) \rangle.
\end{equation}

We recall that a subset $U \subset M$ is called convex if any distance minimizing geodesic between two points in $U$ is contained in $U$. The notion of semiconcavity can be extended to riemannian manifolds as follows.

\begin{definition}
A function $u:U \to \rl$, with $U \subset M$ convex, is called
{\em semiconcave} in $U$ with constant $K$  if, for every geodesic $\gamma:[0,1] \to U$ and $t \in [0,1]$, we have
\begin{equation}\label{eq:semiri}
(1-t) u(\gamma(0)) + t u(\gamma(1)) - u(\gamma(t)) \leq t(1-t)\,K\, \frac{d(\gamma(0),\gamma(1))^2}{2}\,.
\end{equation}
\end{definition}
A detailed exposition of the basic properties of semiconcave functions on a manifold is given in \cite{Villani}. 
Notice that, in such a reference,  functions satisfying \eqref{eq:semiri} are called ``semiconcave with modulus $\omega(t)=Kt^2/2$''. 

It can be checked (see Proposition 10.12 and inequality (10.14) in \cite{Villani}) that, if $u$ is semiconcave with constant $K$, then its superdifferential is nonempty at each point. In addition, any $p \in d^+u(x)$ satisfies
$$
u( \exp_x(v))-u(x) \leq \langle p,v \rangle + K \frac{|v|^2}{2}
$$
for all $x \in U, v \in T_xM$ such that $\exp_x(v) \in U.$

Denote by $\gamma(\cdot)$ the geodesic $\gamma(t)=\exp_x(tv)$ starting at $x$ with speed $v$. If we set $y=\gamma(1)$ and $w=\gamma'(1) \in T_yM$, we have $x=\exp_y(-w)$. Thus, if $q \in d^+u(y)$, we find
$
u(x)-u(y) \leq -\langle q, w \rangle + K |w|^2/2.
$
Since $|w|=d(x,y)=|v|$, we can sum up with the previous inequality to obtain
\begin{equation}\label{e.mon1}
\langle q, w \rangle - \langle p,v \rangle  \leq K |v|^2.
\end{equation}
Denote by $\Pi:T_xM \to T_yM$ the parallel transport along the geodesic $\gamma$. Since parallel transport preserves the scalar product and  $\gamma$ is a geodesic, we have $w=\Pi v$ and $\langle q, w \rangle = \langle \Pi^{-1}(q), v \rangle$.
We conclude that the above inequality can be rewritten as
\begin{equation}\label{monotone}
\langle \Pi^{-1}(q)-p,v \rangle \leq K |v|^2
\end{equation}
for all $v \in T_x M$ with $\exp_x(v) \in U$, and any  $p \in d^+u(x)$ and $q \in d^+u(\exp_x(v))$.

It is well known that the properties of the hessian of the distance function in a riemannian manifold are closely related with the curvature of the manifold. Roughly speaking, positive curvature decreases the hessian of the distance function (i.e., gives a ``stronger'' semiconcavity), while negative curvature increases it. In particular, it can be proved that the square of the distance function is semiconcave with constant $2$ only if the manifold has nonnegative sectional curvature. Even in the case when the curvature has arbitrary sign, however, it turns out that we can replace the square by another function of the distance which enjoys the properties we need for our application. The crucial result for our purposes is the following.

\begin{theorem} \label{semman}
Let $M$ be a riemannian manifold and let $\Omega \subset M$ be any open set (not necessarily smooth). Suppose that all sectional curvatures $\kappa$ at any point of $\Omega$ satisfy $\kappa \geq -\alpha^2$ for some $\alpha>0$ and define 
$v(x)=\cosh (\alpha d_{\partial \Omega}(x))$. Then, given any convex compact set $C \subset \Omega$, the function $v$ is semiconcave on $C$ with constant $K=\alpha^2 \max_{C}v(x)$.
\end{theorem}
\begin{proof} In the case $\alpha=1$ the result is a direct consequence of Lemma 57 in Chapter 11 in \cite{Petersen}, where it is stated that, if the sectional curvatures of $M$ are greater than $-1$ and if $d_p(\cdot)$ is the distance function from a given point $p \in M$, then the function $v(x)=\cosh (d_p(x))$ satisfies $D^2 v(x) \leq v(x) Id$ everywhere. Since $v(x)$ is not everywhere smooth, in general, the above bound on the hessian is understood in a suitable weak form \cite[\S 9.3.1]{Petersen} which easily implies the stated semiconcavity estimate. Once the property is established for $d_p$, it extends to the distance from an arbitrary set because the infimum of semiconcave functions with uniformly bounded semiconcavity constants is semiconcave with the same constant.  The case of a general $\alpha$ is immediately reduced to this one by a homothety of the metric. Finally, we observe that the behaviour of the metric outside $\Omega$ does not influence $d_{\partial \Omega}(x)$ for $x \in \Omega$, and therefore it suffices to assume the bound on the sectional curvature on $\Omega$. 
\end{proof}

\begin{remark}\label{r.semman}
{\rm We mention that, if the infimum of the sectional curvature is zero or positive, then it is possible to use functions different from the hyperbolic cosine which give a sharper semiconcavity estimate (e.g. the square of the distance in the euclidean case); the result of Theorem \ref{semman}, however, suffices for the purposes of this paper. This theorem also imples that the distance function itself is locally semiconcave in $\Omega$. However, the semiconcavity constant in general becomes unbounded as $\partial \Omega$ is approached.
} \end{remark}

We now proceed to show that Theorem \ref{th:pro1} and Corollary \ref{c.prop} can be  extended to manifolds.

\begin{theorem}\label{th:promfd} 
For a given open bounded subset $\Omega \subset M$, let us set $u(x)=d_{\partial \Omega}(x)$ for $x \in \Omega$. For every $x_0\in\Omega$ there exists a
unique lipschitz continuous arc $\gamma :[0,+\infty [\rightarrow \Omega$
such that 
\begin{equation}
\label{eq:gencmfd}
\gamma ' (t)\in d^+u(\gamma(t))\quad t\in [0,+\infty [\text{ a.e.}
\qquad \gamma (0)=x_0.
\end{equation}
The arc $\gamma$ satisfies properties analogous to the ones of Theorem \ref{th:pro1} and Corollary~\ref{c.prop} in the euclidean case. In particular, the right derivative $\gamma'_+ (t)$ exists for every $t \geq 0$, is continuous from the right and satisfies \eqref{eq:gencmfd} everywhere. The derivative of $u$ along $\gamma$ satisfies
\begin{equation}
\label{genc2}
\frac{d}{dt^+} u(\gamma(t)) = |\gamma'_+(t)|^2,\qquad t\in [0,+\infty [\, .
\end{equation}
Moreover, for any $t_0 \geq 0$ such that $\gamma(t_0) \in \Sigma(u)$ there exists $\sigma=\sigma(t_0)>0$ such that $\gamma(t) \in \Sigma(u)$ for all $t \in [t_0,t_0+\sigma[\,$. 
\end{theorem} 
\noindent
Notice that, since $\gamma'(t) \in T_{\gamma(t)}M$ and $d^+u(\gamma(t)) \subset T_{\gamma(t)}^*M$, in (\ref{eq:gencmfd})  the two spaces are identified via the canonical isomorphism. 

\begin{proof} The result can be easily deduced from the euclidean case by using a local coordinate chart. In fact, if $\phi:U \to M$ is a local chart around $x_0$, where $U \subset \rl^n$, and  $G(x)$ is the matrix associated to the scalar product on $T_xM$ in the chart $\phi$, then it is easy to see that  the function $\bar u:=u \circ \phi :U \to \rl$  satisfies
$$
\sum_{i,j=1}^n g^{ij}(x) \frac{\partial \bar u}{\partial x_i}\frac{\partial \bar u}{\partial x_j}=1,
$$
where $g^{ij}(x)$ are the entries of the inverse matrix $G^{-1}(x)$. 
Thus, the assertions of the theorem follow from the corresponding ones of Theorem \ref{th:pro1}. We observe, in particular, that the equation satisfied by the generalized characteristics can be written in local coordinates as
\begin{equation}\label{e:gcr}
\gamma'(t) \in G^{-1} (\gamma(t)) D^+\bar u(\gamma(t)),
\end{equation}
where $D^+ \bar u$ is the euclidean superdifferential of $\bar u$. Finally, the property that $\gamma(t)$ can be defined for $t \in [0,+\infty)$ is obtained by the same argument of Corollary \ref{c.prop}.
  \end{proof}

We now show that Theorem \ref{th:acn}  can be generalized to manifolds.

\begin{theorem}
\label{th:acn2} 
Let $M$ be any smooth complete riemannian manifold, let $\Omega \subset M$ any bounded open set, let $u(\cdot)=d_{\partial \Omega}(\cdot)$ and let $\gamma(\cdot)$ be the arc of Theorem \ref{th:promfd}. If $\gamma(t_0) \in \Sigma(u)$ for some $t_0 \geq 0$ then $\gamma(t) \in \Sigma(u)$ for all $t \in [t_0,+\infty[$.
\end{theorem} 
\begin{proof}
Since $\Omega$ is bounded, we can find a finite value $\alpha>0$ such that the sectional curvature is everywhere greater than $-\alpha^2$ on $\Omega$ and Theorem \ref{semman} can be applied. 

Let $x_0 \in \Omega$ and let $\gamma(t)$ be the solution of the differential inclusion (\ref{eq:gencmfd}). Let us set
$$
p(t)=\gamma'_+(t), \qquad \delta(t)=u(\gamma(t)).
$$
We now suppose  that $\gamma(t_0) \in \Sigma(u)$. We then fix any $s \geq t_0$, and consider $t>s$ close enough to $s$ so that $\gamma(s)$ and $\gamma(t)$ both belong to a neighborhood where any two points are connected by a unique minimal geodesic.

Let us call $v(s,t)$ the vector in $T_{\gamma(s)}M$ such that $\gamma(t)=\exp_{\gamma(s)}(v(s,t))$. Also, we denote by $\Pi_{s,t}:T_{\gamma(s)}M \to T_{\gamma(t)}M$ the isometry induced by the parallel transport along the geodesic connecting $\gamma(s)$ to $\gamma(t)$, and we set $\Pi_{t,s}=\Pi^{-1}_{s,t}$ to denote the inverse map, which is associated to the same geodesic with the opposite direction.

Using Gauss Lemma (see e.g. Lemma 3.3.5 in \cite{DC}), we obtain that
$$
\frac{\partial}{\partial t^+}\,\frac {d(\, \gamma(s) \, , \, \gamma(t) \, )^2}{2} = \langle \Pi_{s,t} v(s,t), \gamma'_+(t) \rangle = \langle  v(s,t), \Pi_{t,s} p(t) \rangle,
$$
and, similarly,
$$
\frac{\partial}{\partial s^+}\,\frac {d(\, \gamma(s) \, , \, \gamma(t) \, )^2}{2} =  \langle -v(s,t), p(s) \rangle\,.
$$
It follows, for $h>0$ small enough,
\begin{multline}\label{dernorm}
\frac{d}{ds^+}  \frac {d(\, \gamma(s) \, , \, \gamma(s+h) \, )^2}{2} = \langle \, \Pi_{s+h,s} p(s+h) - p(s) \, , \,  v(s,s+h) \, \rangle \,.
\end{multline}

Let us now set $\phi(\tau)=\cosh(\alpha \, \tau)$, for $\tau \in \rl$. Then, Theorem \ref{semman} gives a semiconcavity estimate on the function $\phi \circ u$. Let us also observe that,  by the definition of superdifferential,
\begin{equation}\label{e1}
q \in d^+u(x) \ \Longleftrightarrow \ \phi'(u(x)) \, q \in d^+(\phi \circ u)(x).
\end{equation}
Let us denote by $C \subset M$ the spherical neighbourhood of radius $h$ centered at $\gamma(s)$. If $h$ is sufficiently small, $C$ is a convex set which includes the point $\gamma(s+h)$, and in addition $\sup_C u(x) \leq u(\gamma(s))+h$. By Theorem \ref{semman}, we have that  $(\phi \circ u)(x)$ is semiconcave on $C$ with constant given by $K=\alpha^2 \phi (\dt(s)+h)$. 
Therefore, using (\ref{monotone}), (\ref{eq:gencmfd}), \eqref{e1} and the property that $\alpha^2 \phi=\phi''$, we find
\begin{eqnarray*}
&& \langle \, \phi'(\dt(s+h)) \, \Pi_{s+h,s}p(s+h)- \phi'(\delta(s))\, p(s) \, , \, v(s,s+h) \,\rangle \\
&\leq &
\phi'' (\dt(s)+h) \, d(\gamma(s),\gamma(s+h))^2.
\end{eqnarray*}
We rewrite the above inequality as
\begin{eqnarray}
 & & \phi'(\dt(s+h)) \, \langle \Pi_{s+h,s}p(s+h)- p(s),v(s,s+h) \rangle \nonumber \\
 & \leq & \phi'' (\dt(s)+h) \, d(\gamma(s),\gamma(s+h))^2 \label{e7a} \\
& & -[\phi'(\dt(s+h))-\phi'(\dt(s))] \, \langle p(s), v(s,s+h) \rangle. \nonumber
\end{eqnarray}
Let us set
$$
v_h(s)=\frac{v(s,s+h)}{h}, \qquad \delta_h(s) =\frac{\delta(s+h)-\delta(s)}{h}.
$$
Dividing  inequality \eqref{e7a} by $h^2$, we obtain, using formula (\ref{dernorm}) and the fact that $d(\gamma(s),\gamma(s+h))=|v(s,s+h)|=h|v_h(s)|$,
\begin{eqnarray}
\frac{\phi'(\dt(s+h))}2 \frac{d}{ds^+}|v'_h(s)|^2 
& \leq &  \phi'' (\dt(s)+h) |v_h(s)|^2  \nonumber \\ \label{e8} \\
& & -\frac{\phi'(\dt(s+h))-\phi'(\dt(s))}{h} \langle p(s), v_h(s) \rangle. \nonumber
\end{eqnarray}
We rewrite for simplicity this inequality as
\begin{equation}\label{e9}
\frac{d}{ds^+} |v_h(s)|^2 \leq \psi_h(s),
\end{equation}
where we have set
\begin{eqnarray}
\psi_h(s) & = &  \left. \frac{2}{\phi'(\dt(s+h))} \right[ \phi'' (\dt(s)+h) |v_h(s)|^2  \nonumber \\
&& \label {e9a} \\
& &  \left.-\frac{\phi'(\dt(s+h))-\phi'(\dt(s))}{h} \langle p(s), v_h(s) \rangle \right]. \nonumber
\end{eqnarray}
Now, we observe that
\begin{equation}\label{eq:derv}
\lim_{h \downarrow 0} v_h(s)=\gamma'_+(s)=p(s),
\end{equation}
for a.e. $s$, which can be easily checked for instance by using local coordinates around $\gamma(s)$. More formally, the above relation follows from the fact that $\lim_{h \to 0} v_h(s)=\left.\frac{d}{dh}\right|_{h=0} v(s,s+h)$, that $\exp_{\gamma(s)}v(s,s+h)=\gamma(s+h)$ and that the differential of $\exp$ at zero is the identity (see the proof of Proposition 18 in Ch. 5 of \cite{Petersen}); therefore, the derivatives $\left.\frac{d}{dh}\right|_{h=0} v(s,s+h)$ and $\left.\frac{d}{dh}\right|_{h=0} \gamma(s+h)$ coincide.

In addition, using Theorem \ref{th:promfd}, we find that for a.e. $s$
\begin{eqnarray*}
\lefteqn{ \lim_{h \to 0} \frac{\phi'(\dt(s+h))-\phi'(\dt(s))}{h}} \\
& = & \phi''(\dt(s))  \lim_{h \to 0} \frac{\delta(s+h)-\delta(s)}{h} =  \phi''(\dt(s)) |p(s)|^2.
\end{eqnarray*}
Therefore, we see that the function $\psi_h(s)$ defined in \eqref{e9a} is uniformly bounded for $h>0$ small and  $s$ varying in a bounded interval. Moreover, for all $s \geq 0$,
\begin{eqnarray}
\lim_{h \downarrow 0} \psi_h(s) & = & 
2\frac{\phi''(\dt(s))}{\phi'(\dt(s))} \left( |p(s)|^2 - |p(s)|^4 \right) \nonumber \\
&= &
\frac{2 \alpha}{\tanh(\alpha\, \dt(s))} \left( |p(s)|^2 - |p(s)|^4 \right)\,.
\label{e:limphi}
\end{eqnarray}

From this point on, the proof proceeds as in the euclidean case. For $\ep>0$ sufficiently small, we obtain from \eqref{e9}
\begin{equation}\label{e10}
\frac{(|v_h(s)|^2)'}{(|v_h(s)|^2+\ep) (1+\ep-|v_h(s)|^2)} \leq \frac{\psi_h(s)}{(|v_h(s)|^2+\ep) (1+\ep-|v_h(s)|^2)}
\end{equation}
for  a.e. $s \geq t_0$. We assume for simplicity that $t_0=0$. Integrating over $[0,t]$, we have
$$
\log \Big ( \frac{|v_h(s)|^2+\ep}{1+\ep-|v_h(s)|^2}\Big )  \Big|_0^t  \le \Phi_{h,\ep}(t)
$$ 
where 
$$
\Phi_{h,\ep}(t)= (1+2\varepsilon)\int_0^t\frac{\psi_h(s)}{(|v_h(s)|^2+\varepsilon) (1+\varepsilon-|v_h(s)|^2)}\ ds\,.
$$
Therefore, 
\begin{equation}
\label{eq:e5b}
| v_h(t)|^2 \le \frac{e^{\Phi_{h,\ep}(t)} \xi_{h,\ep}(1+\varepsilon)-\varepsilon}{1+e^{\Phi_{h,\ep}(t)} \xi_{h,\ep}}.
\end{equation}
where 
$$
\xi_{h,\ep}:=\frac{|v_h(0)|^2+\varepsilon }{1+\varepsilon-|v_h(0)|^2 }.
$$
Using \eqref{e:limphi} we obtain
\begin{equation*}
\lim_{\ep \downarrow 0} \left( \lim_{h\downarrow 0}\Phi_{h,\ep}\right)=
\int_0^t \frac{2 \alpha}{\tanh(\alpha \delta(s))}\,ds\, =: \beta (t).
\end{equation*}
Letting first $h \downarrow 0$ and then $\ep \downarrow 0$ in \eqref{eq:e5b} we obtain, thanks to \eqref{eq:derv},
$$
|p(t)|^2\le \frac{|p(0)|^2e^{\beta (t)}}{|p(0)|^2 e^{\beta (t)}
  +1-|p(0)|^2},
$$
which implies that $|p(t)|^2<1$ for every $t \geq 0$.
\end{proof}

 \section{Homotopy equivalence}\label{se:homo}

Let us begin by recalling  the well-known notion of homotopy equivalence.
\begin{definition}
\label{def:homt}
Let $X$ and $Y$ be two topological spaces and let  
$$
f:X\to Y\ \text{ and }\ g:X\to Y
$$
be two continuous maps. We say that $f$ and $g$ are homotopic if there
exists a continuous map $H:X\times [0,1]\to Y$, called homotopy,  such that
$$
H(0,\cdot )=f(\cdot
)\ \text{ and }\   H(1,\cdot )=g(\cdot
). 
$$
Furthermore, we say that $X$ and $Y$ have
the same homotopy type if there exist continuous maps
$f:X\to Y$ and $g:Y\to X$ such that $g\circ f$ and $f\circ g$ are
homotopic to the identity on $X$ and $Y$, respectively. 
\end{definition} 
The following result is a direct consequence of Definition \ref{def:homt}. 
\begin{lemma}
\label{lemma:1}
Let $Y\subset X$. If there exists a continuous map
\begin{equation*}
H:X \times [0,1]\to X
\end{equation*}
such that 
\begin{enumerate}

\item[(a)] $H(x,0)=x$, for every $x\in X$, 

\item[(b)] $H(x,1)\in Y$, for every $x\in X$, and

\item[(c)] $H(x,t)\in Y$, for every $(x,t)\in Y\times
  [0,1]$, 

\end{enumerate} 
then $X$ and $Y$ have the same homotopy type. 
\end{lemma}
We are now in a position to apply Theorem~\ref{th:acn2} in order to obtain the following homotopy equivalence result. 
\begin{theorem}
\label{theorem:1}
Let $\Omega$ be a bounded open subset of a smooth riemannian manifold $M$.
Then $\Omega$ has the same  homotopy type as $\Sigma(d_{\partial\Omega})$. 
\end{theorem} 
\noindent
Notice that the above theorem requires no regularity assumption on $\partial \Omega$.
\begin{proof}
In view of Lemma~\ref{lemma:1} it suffices to construct
a continuous map $$H: \Omega \times [0,1]\to \Omega$$
satisfying conditions $(a),(b)$, and $(c)$ above with $Y=\Sigma(d_{\partial\Omega})$.  

For any $x\in\Omega$, let $\gamma_x(\cdot)$ be the generalized characteristic starting at $x$. We claim that 
\begin{equation}
\label{eq:Ts}
\exists\, T>0~:~\forall x\in\Omega \quad \gamma_x(T)\in\Sigma(d_{\partial\Omega}). 
\end{equation}
Indeed, 
set $T=2\diam (\Omega)$ and let 
$x\in \Omega$. Arguing by contradiction, suppose $\gamma_x(T)\notin\Sigma(d_{\partial\Omega})$. Then, 
in light of Theorem~\ref{th:acn2}, $\gamma_x(t)\notin\Sigma(d_{\partial\Omega})$ for all $t\in [0,T]$. So,
$|\gamma'_x(t)|=1$ for every $t\in [0,T]$. Hence, owing to \eqref{genc2}, we have 
\begin{equation*}
d_{\partial \Omega}(\gamma_x(T))=d_{\partial \Omega}(x)+\int_0^T |\gamma'_x(t)|^2\,
dt=d_{\partial \Omega}(x)+T.
\end{equation*}
Then,  
$$
2\diam (\Omega )=T=d_{\partial \Omega}(\gamma_x(T))-d_{\partial \Omega}(x)\le \diam (\Omega) . 
$$
The above contradiction shows that \eqref{eq:Ts} holds true. Next, define 
$$
H(x,t)=\gamma_x(tT)\qquad (x,t)\in \Omega \times  [0,1]. 
$$
We point out that $H$ is a locally Lipschitz continuous map in view of Lemma~\ref{th:unique}. 
Moreover, on account of
(\ref{eq:Ts}) and Theorem~\ref{th:acn2},  
$H$ satisfies conditions $(b)$ and $(c)$ of Lemma~\ref{lemma:1}. 
This completes the proof.  
\end{proof}
The analysis of this paper on riemannian manifolds applies, in particular, to optimal exit time problems in $\R^n$. In the example below, we study a time optimal control problem, deducing a homotopy equivalence result that would be hard to obtain arguing just inside the euclidean framework.
\begin{example}\label{ex:mintime}
Let $F:\rl^n \to \rl^{n \times n}$ be a smooth function such that $\det F(x) \neq 0$ for all $x$ and let $\Omega \subset \rl^n$  be a bounded open set. For any given $x \in \Omega$ we consider the control system
\begin{equation}  \label{e:control}
\left\{
\begin{array}{l}
y'(t)= F(y(t)) \alpha(t), \\
y(0)=x, 
\end{array}\right.
\end{equation}
where $\alpha:[0,+\infty) \to B_1(0)$ is a  measurable function called the control. We denote by $y(\cdot;x,\alpha)$ the trajectory of \eqref{e:control}, and we define the exit time from $\Omega$ of the trajectory as
$$
\tau(x,\alpha)=\inf \{ t>0 ~:~ y(t;x,\alpha) \in \partial \Omega \} \in (0,+\infty].
$$
The {\em minimum time function} is defined as
$$
T(x)=\inf_{\alpha}\tau(x,\alpha), \qquad x \in \Omega.
$$
Under our hypotheses, it is well known that the infimum is attained and that $T(\cdot)$ is a semiconcave solution of the Hamilton--Jacobi--Bellman equation 
\begin{equation}\label{e:hjb}
H(x,DT(x))=1, \qquad x \in \Omega,
\end{equation}
where $H$ is defined as
$$
H(x,p)=\langle F(x) F^*(x) p,p \rangle
$$
with $F^*$ the transpose matrix. 

Let us consider the riemannian metric $g$ on $\rl^n$ induced by the scalar product with matrix
$$G(x):=(F^*)^{-1}(x) F^{-1}(x).$$
Then, using the subscripts $_e$ and $_g$ to distinguish between the euclidean and riemannian metrics, we have
\begin{eqnarray*}
|v|_g \leq 1\; \Longleftrightarrow \;\langle G(x) v, v \rangle_e \leq 1 \; \Longleftrightarrow\; | F^{-1}(x) v |_e \leq 1,
\end{eqnarray*}
which shows that an arc $y(\cdot)$ is an admissible trajectory for the control system \eqref{e:control} if and only $|y'(t)|_g \leq 1$. It follows that $T(x) \equiv d_{\partial \Omega}(x)$, where the distance function $d_{\partial \Omega}$ is taken with respect to the riemannian metric $g$.

Thus, the previous analysis can be applied to the singular set $\Sigma(T)$ of the minimum time function. In particular, recalling also \eqref{e:gcr}, we obtain that $\Sigma(T)$ is invariant under the flow induced by the differential inclusion
$$
\gamma'(t) \in G^{-1}( \gamma(t)) D^+T(\gamma(t)).
$$
The above inclusion, up to a factor $2$, can be written equivalently as  
$$\gamma'(t) \in  D_pH(\gamma(t),D^+T(\gamma(t))),
$$
which is the equation of the characteristics associated with \eqref{e:hjb}.
Moreover, Theorem~\ref{theorem:1} ensures that $\Omega$ and $\Sigma(T)$ have the same homotopy type.
\end{example}



\end{document}